\documentclass[letterpaper, 10 pt, conference]{ieeeconf} 
\IEEEoverridecommandlockouts
\usepackage{cite}
\usepackage{amsmath,amssymb,amsfonts,bm}
\usepackage{algorithmic}
\usepackage{graphicx}
\usepackage{subcaption}
\usepackage{textcomp}
\usepackage{multirow}
\usepackage{xcolor}
\usepackage{physics}
\usepackage[ruled,vlined,onelanguage]{algorithm2e}
\usepackage{array}
\usepackage{fancyhdr}

\newcolumntype{C}[1]{>{\centering\arraybackslash}p{#1}}

\newcommand{\mat}[1]{\left[\begin{matrix} #1 \end{matrix}\right]}

\usepackage{ifthen}
\newboolean{showcomments}
\setboolean{showcomments}{false} 
\ifthenelse{\boolean{showcomments}}{ 
	\newcommand\td[1]{\textcolor{red}{\textbf{TODO:} #1}} 
	\newcommand\tdd[1]{} 
	\newcommand\comment[1]{\textcolor{blue}{\textbf{Comment:} #1}} 
	\newcommand\commentd[1]{} 
	\newcommand\frage[1]{\textcolor{orange}{\textbf{Rückfrage:} #1}} 
	\newcommand\fraged[1]{}
}{ 
	\newcommand\td[1]{}
	\newcommand\tdd[1]{} 
	\newcommand\comment[1]{} 
	\newcommand\commentd[1]{} 
	\newcommand\frage[1]{}
	\newcommand\fraged[1]{}
} 

\newtheorem{definition}{Definition}

\newtheorem{problem}{Problem}
\newtheorem{lemma}{Lemma}

\newtheorem{proposition}{Proposition}
\newtheorem{remark}{Remark}
\newtheorem{assumption}{Assumption}

\title{\LARGE \bf
	Trustworthiness of Optimality Condition Violation in Inverse Dynamic Game Methods Based on the Minimum Principle
}

\author{Philipp Karg$^{1}$, Adrian Kienzle$^{1}$, Jonas Kaub$^{1}$, Balint Varga$^{1}$ and Sören Hohmann$^{1}$
	\thanks{$^{1}$All authors are with the Institute of Control Systems (IRS),
		Karlsruhe Institute of Technology (KIT), 76131 Karlsruhe, Germany. Corresponding author is Philipp Karg, {\tt\small philipp.karg@kit.edu}.}%
	\thanks{This work was funded by the Deutsche Forschungsgemeinschaft (DFG, German Research Foundation) - 421428832.}
}

\begin{document}

\maketitle
\thispagestyle{fancy}
\pagestyle{fancy}

\fancyhf{}
\fancyhead[CO,CE]{\copyright 2024 IEEE. This paper has been accepted for publication and presentation at the \\8th IEEE Conference on Control Technology and Applications.}

\begin{abstract}
	In this work, we analyze the applicability of Inverse Dynamic Game (IDG) methods based on the Minimum Principle (MP). The IDG method determines unknown cost functions in a single- or multi-agent setting from observed system trajectories by minimizing the so-called residual error, i.e. the extent to which the optimality conditions of the MP are violated with a current guess of cost functions. The main assumption of the IDG method to recover cost functions such that the resulting trajectories match the observed ones is that the given trajectories are the result of a Dynamic Game (DG) problem with known parameterized cost function structures. However, in practice, when the IDG method is used to identify the behavior of unknown agents, e.g. humans, this assumption cannot be guaranteed. Hence, we introduce the notion of the trustworthiness of the residual error and provide necessary conditions for it to define when the IDG method based on the MP is applicable to such problems. From the necessary conditions, we conclude that the MP-based IDG method cannot be used to validate DG models for unknown agents but can yield under certain conditions robust parameter identifications, e.g. to measurement noise. Finally, we illustrate these conclusions by validating a DG model for the collision avoidance behavior between two mobile robots with human operators. 
\end{abstract}


\section{Introduction} \label{sec:introduction}
Inverse Optimal Control (IOC) methods have gained significant research interest in the last years. Starting with the work of Kalman \cite{Kalman.1964}, the question is asked whether a given arbitrary control law is an optimal one, i.e. the solution to an Optimal Control (OC) problem. To answer this research question, conditions on the system dynamics, control laws and cost functions were proposed under which the given control law is optimal\comment{often no constructive procedure to compute the cost functions actually}. In recent years, so-called data-based IOC methods arose (see e.g. \cite{Molloy.2018,Kamalapurkar.2018,Panchea.2017,Aghasadeghi.2014,Johnson.2013,Mombaur.2010}) which aim at computing unknown cost functions from given system trajectories such that these trajectories are equal to the optimal trajectories resulting from the determined cost functions. In addition, these methods were extended to inverse coupled dynamic optimization problems, i.e. Inverse Dynamic Games (IDG), to determine cost functions for all players such that the given system trajectories are equal to the Nash trajectories resulting from the computed cost functions (see e.g. \cite{Molloy.2022,Inga.2021,Molloy.2020,Inga.2020,Awasthi.2020,Inga.2019b,Molloy.2017}). 

There are two kinds of methods to solve such data-based inverse dynamic optimization problems. In direct approaches (see e.g. \cite{Molloy.2022,Molloy.2017,Mombaur.2010}), the error between the given system trajectories (so-called Ground Truth (GT)) and trajectories resulting from a current guess of the cost functions is minimized. Here, typically bi-level optimization problems follow. In an upper level, the trajectory error is minimized and in a lower level, a dynamic optimization problem is solved to evaluate the trajectory error in the upper level at a current cost function guess. In indirect approaches (see e.g. \cite{Molloy.2022,Inga.2021,Molloy.2020,Inga.2020,Awasthi.2020,Inga.2019b,Molloy.2017,Molloy.2018,Kamalapurkar.2018,Panchea.2017,Aghasadeghi.2014,Johnson.2013}), an important assumption, which is made on the given trajectories (see e.g. \cite[Definition~6.1]{Molloy.2022}, \cite[Assumption~1]{Molloy.2020} or \cite[Assumption~4.2]{Inga.2020}), is that they constitute a solution to a Dynamic Game (DG) problem with parameterized cost functions, e.g. running costs are represented by a linear combination of basis functions\comment{with this assumption, an exact solution for the direct approaches is guaranteed if the known parameterized cost functions are trained and if the optimization algorithm used for training yields a global minimum of the non-convex objective function based on the trajectory error - if the assumption is not fulfilled, such an optimization algorithm still yields the best possible solution with the used parameterized cost function structure (very general approximation structures for the cost function can be used in principle) - however, due to non-convex objective function and typical bi-level structure we have high computation times}. Via the indirect method, the unknown parameters of these parameterized cost functions are computed by minimizing violations of optimality conditions, i.e. the so-called residual error, that are fulfilled for the given trajectories and the unknown optimal parameters\comment{convex optimization problems can be derived that enable fast/real-time capable and even online computations as well as constitution of solution sets - however, the question is how we get the parameterized cost function structure in practice and what happens if the assumption is not fulfilled, i.e. can we still find "usable" parameters that yield sufficiently small trajectory errors (robust parameter identification)}. The residual error can be defined based on the Minimum Principle (MP), like in \cite{Molloy.2022,Molloy.2020,Inga.2020,Awasthi.2020,Molloy.2017,Molloy.2018,Johnson.2013}, the Hamilton-Jacobi-Bellman equations (see e.g. \cite{Inga.2021,Inga.2019b,Kamalapurkar.2018}), the Euler-Lagrange equations (see e.g. \cite{Aghasadeghi.2014}) or the Karush-Kuhn-Tucker conditions (see e.g. \cite{Panchea.2017})\comment{currently, no concrete formal statement that under the main assumption and perhaps further conditions the minimizers of the residual error solve the IDG problem actually - however, set of sufficient conditions can be derived from literature: if main assumption holds and if uniqueness condition (rank condition) (see e.g. \cite{Molloy.2022,Molloy.2020,Inga.2020,Molloy.2018,Aghasadeghi.2014}) hold, all parameters derived from the residual optimization are scaled versions of the unknown (but existing) optimal parameters and hence leading to zero trajectory error}.

\tdd{reformulate and clarify that the indirect methods solve another IOC problem (parameters for known model structure - see Problem 1) than typically investigated in practice and as described below}
The main assumption of residual-based IDG methods that the given trajectories are optimal w.r.t. known parameterized cost function structures is to be seen critical in most applications. Data-based IDG approaches are typically used to identify the behavior of an unknown agent or unknown mutually coupled agents from measurement data. Examples can be found in \cite{Westermann.2020,Jin.2019,Berret.2011,Mombaur.2010} for human movement identification, in \cite{Inga.2020,Inga.2019} for the cooperative steering behavior of two humans or in \cite{Molloy.2018b} for the collision avoidance behavior between two birds. In all these cases, the first question to be examined is "are the observed trajectories optimal", i.e. can a DG model be found to describe the GT data. Since the cost function structure is unknown as well in such cases, fulfillment of the main assumption cannot be guaranteed. Hence, the first crucial question for the applicability of residual-based IDG methods arises: can they be used to validate DG models at all. Then, the second question for the applicability of residual-based IDG methods is under which conditions they guarantee a robust parameter identification. Since an exact match with the GT trajectories is unrealistic even with the best DG model due to disturbances like measurement noise, the residual-based IDG method should still find these best possible parameters. For a residual-based IDG method to possess these characteristics, what we call in the following the trustworthiness of the residual error is necessary, i.e. global minimizers of the residual error yield trajectories with global minimum errors to the GT trajectories. With necessary conditions for the trustworthiness, the applicability of residual-based IDG methods for model validation and robust parameter identification can be discussed. Until now, research regarding the trustworthiness of residual errors is scarce. In literature (see e.g. \cite{Inga.2020,Inga.2019b,Molloy.2017,Johnson.2013}), only simulations with noisy GT trajectories, which result from trajectories that are optimal w.r.t. an unknown parameterization of a known cost function structure but disturbed by additive Gaussian noise, can be found. Here, the residual error seems to be approximately trustworthy, but conditions or a deeper analysis, e.g. for unknown cost function structures, is missing\comment{only result in \cite{Inga.2020} indicates that in case of wrong/unknown basis functions the trustworthiness cannot be given}.

With the contributions in this paper, we close this research gap. Firstly, we define the trustworthiness of the residual error for IDG methods based on the MP. We show that the main assumption of the IDG method (optimality of the given trajectories w.r.t. known cost function structures) is actually a necessary condition for the trustworthiness. Therefore, we conclude that the MP-based IDG method cannot be used in practice to validate a DG model. However, we further conclude that if e.g. via a direct IDG method cost function structures are found that guarantee the best possible fit of a DG model to the GT trajectories, the MP-based IDG method is robust to e.g. noisy GT data. Finally, we illustrate these conclusions in a practical example by validating a differential game model for the collision avoidance behavior between two mobile robots with human operators. 


\section{Inverse Dynamic Game Method based on the Minimum Principle} \label{sec:IOC_minimum_principle}

\tdd{shorten if possible}
In this section, we introduce the IDG method based on the MP \cite{Molloy.2022,Molloy.2020,Inga.2020,Molloy.2017}. 

Let an input-affine dynamic system be defined by
\begin{align} \label{eq:dynamic_system}
	\begin{aligned}
		\dot{\bm{x}} = \bm{f}\left( \bm{x}, \bm{u}_1, \dots, \bm{u}_N \right) = \bm{f}_{\bm{x}}(\bm{x}) + \sum_{i=1}^{N} \bm{G}_i (\bm{x}) \bm{u}_i
	\end{aligned}
\end{align}
and $\bm{x}(0) = \bm{x}_0$, where $\bm{x} \in \mathbb{R}^n$ denotes the system state, $\bm{u}_i \in \mathbb{R}^{m_i}$ the control variable of player $i \in \mathcal{P} = \{1,\dots,N\}$ and $\bm{x}_0$ the initial state at the initial time $t_0 = 0$. Moreover, $\bm{f}_{\bm{x}}: \mathbb{R}^n \rightarrow \mathbb{R}^n$ and $\bm{G}_i: \mathbb{R}^n \rightarrow \mathbb{R}^{n\times m_i}$ are continuously differentiable w.r.t. $\bm{x}$. Furthermore, let $\bm{u} = \mat{ \bm{u}_1^\top & \dots & \bm{u}_N^\top }^\top$ and $\bm{u}_{\neg i} = \mat{ \bm{u}_1^\top & \dots & \bm{u}_{i-1}^\top & \bm{u}_{i+1}^\top & \dots & \bm{u}_N^\top }^\top$. In the DG, each player $i \in \mathcal{P}$ influences system~\eqref{eq:dynamic_system} by applying an open-loop control strategy $\bm{u}_i(t) = \bm{\gamma}_i\left( \bm{x}_0, t \right), \forall t \in [0,T]$ from the set of admissible strategies $\Gamma_i$ such that its individual cost function
\begin{align} \label{eq:cost_function}
	J_i = h_i(\bm{x}(T)) + \int_{0}^{T} g_i\left( \bm{x}, \bm{u}_i \right) \text{d}t
\end{align}
is minimized, where $h_i: \mathbb{R}^n \rightarrow \mathbb{R}$ denotes the terminal costs and $g_i: \mathbb{R}^n \times \mathbb{R}^{m_i} \rightarrow \mathbb{R}$ the running costs. Furthermore, $h_i$ is assumed to be continuously differentiable and convex w.r.t. $\bm{x}$ and $g_i(\bm{x},\bm{u}_i) = q_i(\bm{x}) + \bm{u}_i^\top \bm{R}_i \bm{u}_i$ with $q_i: \mathbb{R}^n \rightarrow \mathbb{R}$ being continuously differentiable w.r.t. $\bm{x}$ and $\bm{R}_i$ positive definite. 

We assume a non-cooperative game and thus, the corresponding DG solution concept are Open-Loop Nash equilibria (OLNE) (cf.~\cite[p.~266]{Basar.1999}). Lemma~\ref{lemma:OLNE} states the necessary and sufficient conditions to calculate them.

\begin{lemma} \label{lemma:OLNE}
	Let a DG be defined by \eqref{eq:dynamic_system} and \eqref{eq:cost_function}. Furthermore, let the Hamiltonian functions
	\begin{align} \label{eq:hamiltonian}
		H_i(\bm{\psi}_i,\bm{x},\bm{u}) = g_i(\bm{x},\bm{u}_i) + \bm{\psi}_i^\top \bm{f}(\bm{x},\bm{u}), \quad \forall i \in \mathcal{P},
	\end{align}
	where $\bm{\psi}_i: [0,T] \rightarrow \mathbb{R}^n$ are so-called costate functions, be continuously differentiable and convex w.r.t. $\bm{x}$. If $\bm{\gamma}^*(\cdot) = \mat{ {\bm{\gamma}_1^*(\bm{x}_0,\cdot)}^\top & \dots & {\bm{\gamma}_N^*(\bm{x}_0,\cdot)}^\top }^\top = \bm{u}^*(\cdot)$ provides an OLNE with the corresponding state trajectory $\bm{x}^*(\cdot)$, the costate functions $\bm{\psi}_i(\cdot), \forall i \in \mathcal{P}$ fulfill
	\begin{align}
		\dot{\bm{x}}^*(t) &= \bm{f}\left( \bm{x}^*(t), \bm{u}^*(t) \right), \quad \bm{x}^*(0) = \bm{x}_0, \label{eq:OLNE_1} \\
		\bm{0} &= \nabla_{\bm{u}_i} H_i\left( \bm{\psi}_i(t), \bm{x}^*(t), \bm{u}_i, \bm{u}_{\neg i}^*(t) \right) \big \vert_{\bm{u}_i = \bm{u}_i^*(t)}, \label{eq:OLNE_2} \\
		\dot{\bm{\psi}}_i(t) &= -\nabla_{\bm{x}} H_i(\bm{\psi}_i(t), \bm{x}, \bm{u}^*(t)) \big \vert_{\bm{x} = \bm{x}^*(t)}, \label{eq:OLNE_3} \\
		\bm{\psi}_i(T) &= \nabla_{\bm{x}} h_i\left( \bm{x} \right) \big \vert_{\bm{x} = \bm{x}^*(T)}. \label{eq:OLNE_4}
	\end{align}
	Furthermore, if a set of costate functions $\bm{\psi}_i(\cdot), \forall i \in \mathcal{P}$, control $\bm{u}^*(\cdot)$ and state trajectories $\bm{x}^*(\cdot)$ satisfy \eqref{eq:OLNE_1}, \eqref{eq:OLNE_2}, \eqref{eq:OLNE_3} and \eqref{eq:OLNE_4}, $\bm{u}^*(\cdot)$ constitutes an OLNE.
\end{lemma}   
\begin{proof}
	Necessity of the conditions \eqref{eq:OLNE_1}, \eqref{eq:OLNE_2}, \eqref{eq:OLNE_3} and \eqref{eq:OLNE_4} for an OLNE follows from \cite[Theorem~6.11]{Basar.1999} and sufficiency from \cite[Theorem~3.2]{Dockner.2000}\tdd{check theorem}.
\end{proof}


In order to define and solve the inverse problem to the introduced DG, a common approach is to rewrite $h_i$ and $g_i$ in \eqref{eq:cost_function} as linear combination of basis functions: $h_i(\bm{x}) = \bm{\zeta}_i^\top \bm{\lambda}_i(\bm{x})$ and $g_i(\bm{x},\bm{u}_i) = \bm{\eta}_i^\top \bm{\mu}_i(\bm{x},\bm{u}_i)$. This yields
\begin{align} \label{eq:cost_function_basisFunctions}
	J_i \widehat{=} \int_{0}^{T} \bm{\theta}_i^\top \bm{\phi}_i\left( \bm{x}, \bm{u} \right) \text{d}t
\end{align}
with $\bm{\theta}_i^\top = \mat{ \bm{\eta}_i^\top & \bm{\zeta}_i^\top } \in \mathbb{R}^{M_i}$ and $\bm{\phi}_i^\top = \mat{ \bm{\mu}_i^\top & \dot{\bm{x}}^\top \pdv{\bm{\lambda}_i}{\bm{x}}^\top }$. 

Now, Assumption~\ref{assumption:GT_data} is the typical assumption (see e.g. \cite[Definition~6.1]{Molloy.2022}, \cite[Assumption~1]{Molloy.2020} or \cite[Assumption~4.2]{Inga.2020}) made on the GT data to define the IDG Problem~\ref{problem:IDG} for the IDG methods based on the MP.
\begin{assumption} \label{assumption:GT_data}
	Let observed control $\tilde{\bm{u}}_i(\cdot), \forall i \in \mathcal{P}$ and state trajectories $\tilde{\bm{x}}(\cdot)$ be given (so-called GT data). Moreover, let these observed trajectories constitute an OLNE of the DG defined by \eqref{eq:dynamic_system} and \eqref{eq:cost_function_basisFunctions} with $\bm{\theta}_i^*, \forall i \in \mathcal{P}$: $\tilde{\bm{u}}_i(\cdot) = \bm{u}_i^*(\cdot), \forall i \in \mathcal{P}$ and $\tilde{\bm{x}}(\cdot) = \bm{x}^*(\cdot)$.
\end{assumption}
\begin{problem} \label{problem:IDG}
	Let Assumption~\ref{assumption:GT_data} hold.
	Find (at least) one non-trivial parameter vector $\hat{\bm{\theta}}_i, \forall i \in \mathcal{P}$ such that
	\begin{align} \label{eq:IDG_problem}
		\begin{aligned}
			\tilde{\bm{u}}_i(\cdot) &= \arg \min_{\bm{u}_i(\cdot)} \int_{0}^{T} \hat{\bm{\theta}}_i^\top \bm{\phi}_i(\bm{x},\bm{u}) \text{d}t, \quad \forall i \in \mathcal{P} \\
			\text{w.r.t.} &\hphantom{=} \dot{\bm{x}} = \bm{f}(\bm{x},\bm{u}), \bm{x}(0)=\bm{x}_0.
		\end{aligned}
	\end{align}
\end{problem}

Due to Assumption~\ref{assumption:GT_data}, the trajectories $\tilde{\bm{u}}(\cdot)$ and $\tilde{\bm{x}}(\cdot)$ fulfill the optimality conditions~\eqref{eq:OLNE_1}, \eqref{eq:OLNE_2}, \eqref{eq:OLNE_3} and \eqref{eq:OLNE_4} with $\bm{\theta}_i^*, \forall i \in \mathcal{P}$. Since the parameters $\bm{\theta}_i^*$ are unknown in case of Problem~\ref{problem:IDG}, the aim is to find $\hat{\bm{\theta}}_i, \forall i \in \mathcal{P}$ such that \eqref{eq:OLNE_2} and \eqref{eq:OLNE_3} are satisfied. This can be done by minimizing the extent to which these OLNE conditions are violated. This extent can be formulated by defining the residual functions $r_{C,i}(\bm{\theta}_i,\bm{\psi}_i) = \norm{\nabla_{\bm{u}_i}H_i(\bm{\theta}_i,\bm{\psi}_i)}_2^2$ and $r_{L,i}(\bm{\theta}_i,\bm{\psi}_i) = \norm{\dot{\bm{\psi}}_i+\nabla_{\bm{x}}H_i(\bm{\theta}_i,\bm{\psi}_i)}_2^2$ for each player $i \in \mathcal{P}$. Minimizing the so-called residual error~$\delta_{R,i}$ w.r.t. $\bm{\theta}_i$,
\begin{align} \label{eq:residual_error_opt}
	\min_{\bm{\theta}_i,\bm{\psi}_i(\cdot)} \underbrace{\int_{0}^{T} r_{C,i}(\bm{\theta}_i,\bm{\psi}_i) + r_{L,i}(\bm{\theta}_i,\bm{\psi}_i) \text{d}t}_{\delta_{R,i}(\bm{\theta}_i,\bm{\psi}_i)},
\end{align}
yields $\hat{\bm{\theta}}_i$. According to \cite[Lemma~4.2]{Inga.2020}, \eqref{eq:residual_error_opt} can be rewritten as a Quadratic Program (QP), which can be solved independently for each player $i \in \mathcal{P}$.
\begin{lemma} \label{lemma:residual_error_QP}
	The residual error optimization~\eqref{eq:residual_error_opt} is solved by $\hat{\bm{\theta}}_i$ and $\hat{\bm{\psi}}_i(\cdot)$ if and only if $\hat{\bm{\alpha}}_i^\top = \mat{ \hat{\bm{\theta}}_i^\top & \hat{\bm{\psi}}^\top_i(0) }$ solves
	\begin{align} \label{eq:residual_error_QP}
		\min_{\bm{\alpha}_i} \underbrace{\bm{\alpha}_i^\top \bm{P}_i(0) \bm{\alpha}_i}_{\delta_{R,i}(\bm{\alpha}_i)},
	\end{align}
	where $\bm{P}_i(0)$ is symmetric, positive semidefinite and follows from the Riccati differential equation
	\begin{align} \label{eq:residual_error_RiccatiEq}
		\dot{\bm{P}}_i(t) = \left( \bm{P}_i(t)\bm{B}_i \! + \! \bm{N}_i(t) \right)\left( \bm{P}_i(t)\bm{B}_i \! + \! \bm{N}_i(t) \right)^\top \!\! - \! \bm{F}^\top_i(t)\bm{F}_i(t)
	\end{align}
	with $\bm{P}_i(T)=\bm{0}$, $\bm{B}_i = \mat{ \bm{0}_{n\times M_i} & \bm{I}_{n} }^\top$, $\bm{N}^\top_i(t) = \mat{ \pdv{\bm{\phi}_i(t)}{\bm{x}}^\top & \pdv{\bm{f}(t)}{\bm{x}}^\top }$ and
	\begin{align} \label{eq:F_i}
		\bm{F}_i(t) = \mat{ \pdv{\bm{\phi}_i(t)}{\bm{x}}^\top & \pdv{\bm{f}(t)}{\bm{x}}^\top \\ \pdv{\bm{\phi}_i(t)}{\bm{u}_i}^\top & \pdv{\bm{f}(t)}{\bm{u}_i}^\top }.
	\end{align}
\end{lemma}
\begin{proof}
	The result follows from \cite[Lemma~4.1]{Inga.2020} and \cite[Lemma~4.2]{Inga.2020}.
\end{proof}

By solving the convex QP~\eqref{eq:residual_error_QP} with additional constraints, e.g. $\theta_{i,1}=1$ to avoid the trivial solution, for each player $i \in \mathcal{P}$, we find a global minimizer $\hat{\bm{\theta}}^{R}{}^{\top} = \mat{ \hat{\bm{\theta}}^R_1{}^{\top} & \dots & \hat{\bm{\theta}}^R_N{}^{\top} }$ of the residual error~$\delta_{R}(\bm{\alpha}) = \sum_{i=1}^{N} \delta_{R,i}(\bm{\alpha}_i)$~\eqref{eq:residual_error_opt} ($\bm{\alpha}^\top = \mat{ \bm{\alpha}_1^\top & \dots & \bm{\alpha}_N^\top }$). However, Problem~\ref{problem:IDG} aims at finding a global minimizer $\hat{\bm{\theta}}^{T}$ of a trajectory error measure $\delta_{T}(\bm{\theta})=\delta_{T}^{\bm{x}}(\bm{\theta})+\delta_{T}^{\bm{u}}(\bm{\theta})$ between the GT $\tilde{\bm{u}}(\cdot)$ and $\tilde{\bm{x}}(\cdot)$ and the estimated trajectories $\hat{\bm{u}}(\cdot)$ and $\hat{\bm{x}}(\cdot)$ such that $\delta_{T}(\hat{\bm{\theta}}^T) = 0$. The trajectory error can for example be quantified by the normalized sum of absolute errors (NSAE)\footnote{For the NSAE, the trajectories are evaluated at $K$ points in time $t_{k} \in [0,T]$.}:
\begin{align}
	\delta_{T}^{\bm{x}}(\bm{\theta}) &= \sum_{j=1}^{n} \frac{1}{\max_{k} \abs{\tilde{x}_j^{(k)}}} \sum_{k=1}^{K} \abs{ \tilde{x}_j^{(k)} - x_j^{(k)} \big \vert_{\bm{\theta}} }, \label{eq:trajectory_error_x} \\
	\delta_{T}^{\bm{u}}(\bm{\theta}) &= \sum_{i=1}^{N} \sum_{j=1}^{m_i} \frac{1}{\max_{k} \abs{\tilde{u}_{i,j}^{(k)}}} \sum_{k=1}^{K} \abs{ \tilde{u}_{i,j}^{(k)} - u_{i,j}^{(k)} \big \vert_{\bm{\theta}} }, \label{eq:trajectory_error_u}
\end{align}
where $\bm{u}(\cdot) \big \vert_{\bm{\theta}}$ is the OLNE of the DG~\eqref{eq:dynamic_system}, \eqref{eq:cost_function_basisFunctions} with $\bm{\theta}$ and $\bm{x}(\cdot) \big \vert_{\bm{\theta}}$ the corresponding state trajectory.
\comment{compared to indirect IDG method, direct/bi-level-based method directly minimizes this trajectory error and yield $\hat{\bm{\theta}}^T$ if the non-convexity of the objective function is handled}

In references on indirect IDG methods, often a formal statement about the connection of the residual~$\hat{\bm{\theta}}^R$ and the trajectory error minimizer~$\hat{\bm{\theta}}^T$ is missing. Lemma~\ref{lemma:sufficient_conditions_trustworthiness} restates sufficient conditions for a unique solution of the residual parameter optimizations and clarifies that these conditions are sufficient for $\delta_T(\hat{\bm{\theta}}^R)=0$.
\begin{lemma} \label{lemma:sufficient_conditions_trustworthiness}
	Let $\bar{\bm{P}}_i$ follow from $\bm{P}_i(0)$ by deleting the first row and column. Furthermore, let $\bar{\bm{p}}_i$ be the first column of $\bm{P}_i(0)$ without the first element and $\bar{\bm{P}}^+_i$ the pseudoinverse of $\bar{\bm{P}}_i$. Let the singular value decomposition of $\bar{\bm{P}}_i$ be given by $\bar{\bm{P}}_i = \bm{U}_i \bm{\Sigma}_i \bm{U}_i^\top$ with
	\begin{align} \label{eq:U_i}
		\bm{U}_i = \mat{ \bm{U}_i^{11} & \bm{U}_i^{12} \\ \bm{U}_i^{21} & \bm{U}_i^{22} },
	\end{align}
	where $\bm{U}_i^{11} \in \mathbb{R}^{(M_i-1)\times r_i}$, $\bm{U}_i^{12} \in \mathbb{R}^{(M_i-1)\times (M_i+n-1-r_i)}$, $\bm{U}_i^{21} \in \mathbb{R}^{n\times r_i}$, $\bm{U}_i^{22} \in \mathbb{R}^{n\times (M_i+n-1-r_i)}$ and $r_i$ is the rank of $\bar{\bm{P}}_i$.
	If Assumption~\ref{assumption:GT_data} holds and $r_i=M_i+n-1$ or $\bm{U}_i^{12}=\bm{0}$ holds $\forall i \in \mathcal{P}$, $\hat{\bm{\theta}}^R_i = c_i \bm{\theta}_i^*$, $\forall i \in \mathcal{P}$ with $c_i \in \mathbb{R}_{>0}$ and hence, the parameters $\hat{\bm{\theta}}^R_i$ solve Problem~\ref{problem:IDG}, i.e. $\delta_T(\hat{\bm{\theta}}^R)=0$.
\end{lemma}
\begin{proof}
	From \cite[Theorem~2]{Molloy.2020}, all solutions~$\hat{\bm{\theta}}^R_i$ of \eqref{eq:residual_error_QP} are unique up to a scaling factor $c_i \in \mathbb{R}_{>0}$ if $r_i=M_i+n-1$ or $\bm{U}_i^{12}=\bm{0}$. Furthermore, since Assumption~\ref{assumption:GT_data} holds and the MP conditions are necessary, for all global minimizers of \eqref{eq:residual_error_QP} $\delta_{R}=0$ holds and $\bm{\theta}_i^*$ is one of them. Thus, we find a $c_i$ such that $\hat{\bm{\theta}}^R_i = c_i \bm{\theta}_i^*$. Lastly, all parameters $\hat{\bm{\theta}}^R_i =  c_i\bm{\theta}_i^*, \forall i \in \mathcal{P}$ yield $\hat{\bm{u}}(\cdot) = \tilde{\bm{u}}(\cdot)$ and $\hat{\bm{x}}(\cdot) = \tilde{\bm{x}}(\cdot)$. 
\end{proof}
\comment{conditions sufficient for trustworthiness of $\delta_R$ with zero trajectory error; relaxed sufficient conditions for that possible, max rank for $\bm{P}_i(0)$ (excitation condition), solution to QP such that MP conditions are still necessary and (uniquely) sufficient and Assumption~\ref{assumption:GT_data}; then, solution set via global minima of QP}

As explained in Section~\ref{sec:introduction}, Assumption~\ref{assumption:GT_data} cannot be guaranteed in applications where the behavior of unknown agents needs to be identified. The parameterized cost function structures, i.e. the basis functions~$\bm{\phi}_i$ in \eqref{eq:cost_function_basisFunctions}, are unknown as well. In this case, the connection between the minimizers $\hat{\bm{\theta}}^R$ and $\hat{\bm{\theta}}^T$ is unclear so far but crucial for model validation and robust parameter identification with the MP-based IDG method. For example, if the global residual $\hat{\bm{\theta}}^R$ and the global trajectory error minimizers $\hat{\bm{\theta}}^T$ decouple\comment{at least one is different}, falsification of a model assumption would never be possible since $\delta_T(\hat{\bm{\theta}}^T)=0$ can be possible but $\hat{\bm{\theta}}^R$ always results in $\delta_T(\hat{\bm{\theta}}^R)>0$. Furthermore, if the minimizers of both error types differ\comment{at least one}, it is not guaranteed that by adapting the cost function structures or by using overparameterized ones $\delta_T(\hat{\bm{\theta}}^R)=0$ can be achieved although $\delta_T(\hat{\bm{\theta}}^T)=0$ holds. In addition, if Problem~\ref{problem:IDG} can only be solved approximately (normally the case in practice), i.e. $\delta_T(\hat{\bm{\theta}}^T)<\epsilon_T$, where $\epsilon_T$ is small enough for the concrete application, an approximate match between $\hat{\bm{\theta}}^R$ and $\hat{\bm{\theta}}^T$ is important for a robust parameter identification by the MP-based IDG method. Although in such cases an exact match $\hat{\bm{\theta}}^R = \hat{\bm{\theta}}^T$ is not possible, in $\norm{\hat{\bm{\theta}}^R - \hat{\bm{\theta}}^T}<\epsilon$, $\epsilon$ should be sufficiently small such that the application requirement $\delta_T(\hat{\bm{\theta}}^R)<\epsilon_T$ is met. 

In the next section, we formalize the connection between $\hat{\bm{\theta}}^R$ and $\hat{\bm{\theta}}^T$ by defining the so-called trustworthiness of the residual error~$\delta_{R}$ and give necessary conditions for it to discuss the applicability of the MP-based IDG method for model validation and robust parameter identification.



\section{Necessary Conditions for the Trustworthiness of Optimality Condition Violation} \label{sec:optimalityConditionViolation}
\tdd{reformulate with $\epsilon = 0$ and the complete section only for MP}
\begin{definition}[Trustworthiness] \label{def:trustworthiness_residual_error}
	The residual error $\delta_{R}(\bm{\alpha})$ of the IDG method based on the MP is trustworthy if a global minimizer $\hat{\bm{\theta}}^{R}$ of $\delta_{R}(\bm{\alpha})$ is a global minimizer $\hat{\bm{\theta}}^{T}$ of $\delta_T(\bm{\theta})=\delta_T^{\bm{x}}(\bm{\theta})+\delta_T^{\bm{u}}(\bm{\theta})$.
	\comment{due to the convex QP, the set of global minimizers of the residual error is non-empty - hence, the residual-based optimization finds in case of trustworthiness a subset of the global minimizers of the trajectory error; since $\delta_{T}$ can be computed at the global minimizers of $\delta_{R}$ the numerical values of the global minima of $\delta_R$ are not important}
\end{definition}

\begin{remark} 
	In Lemma~\ref{lemma:sufficient_conditions_trustworthiness} sufficient conditions for the trustworthiness in the sense of Definition~\ref{def:trustworthiness_residual_error} are provided.
\end{remark}


In order to illustrate the trustworthiness of the residual error, we first look at the case where the conditions of Lemma~\ref{lemma:sufficient_conditions_trustworthiness} hold. Let the OC problem of a single-player double integrator be defined by the dynamic system
\begin{align} \label{eq:sys_doubleInt}
	\dot{\bm{x}} = \mat{0 & 1 \\ 0 & 0}\bm{x} + \mat{0 \\ 1}u, \quad \bm{x}(0) = \mat{1 \\ -1}
\end{align}
and the cost function
\begin{align} \label{eq:cost_doubleInt}
	J_i = \int_{0}^{6} \mat{ 1 & 2 & 1 } \mat{u^2 & x_1^2 & x_2^2 }^\top \text{d}t.
\end{align}
Solving the OC problem
given by \eqref{eq:sys_doubleInt}, \eqref{eq:cost_doubleInt} yields the GT trajectories $\tilde{\bm{x}}(\cdot)$ and $\tilde{u}(\cdot)$ that are optimal w.r.t. $\bm{\theta}^* = \mat{ 1 & 2 & 1 }^\top$. 
By evaluating the residual~$\delta_R(\bm{\alpha})$ and the trajectory error~$\delta_{T}(\bm{\theta})$ for varying $\theta_2$-values by setting $\theta_1=\theta^*_1$, $\theta_3=\theta_3^*$ and $\bm{\psi}(0)=\bm{\psi}^*(0)$, Fig.~\ref{fig:comparisonErrorsNoiseFree} results. Here, the residual error~$\delta_{R}$ is trustworthy in the sense of Definition~\ref{def:trustworthiness_residual_error}: the global minimizer of $\delta_{R}$ is also a global minimizer of $\delta_T$ and since for the global minimum trajectory error $\delta_{T}=0$ holds, they solve Problem~\ref{problem:IDG}.
\begin{figure}[t]
	\centering
	\includegraphics[width=2.5in]{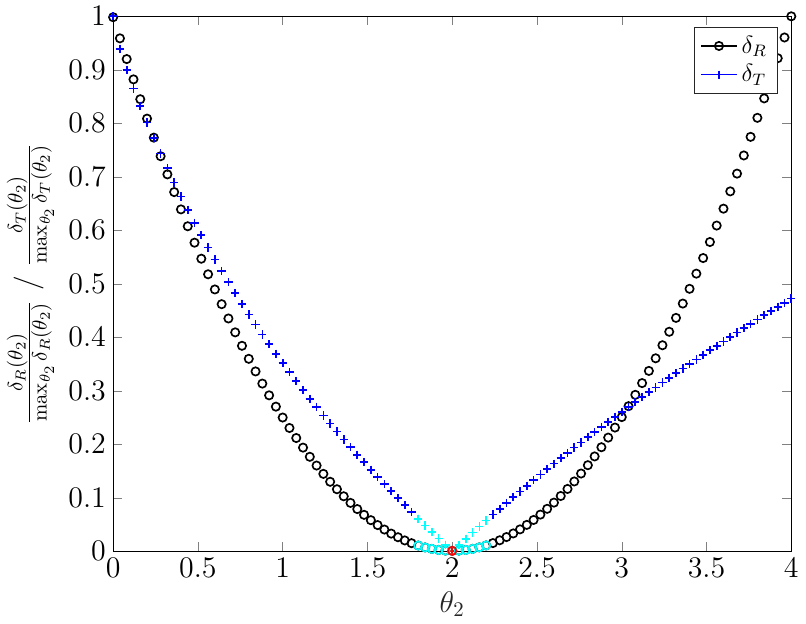}
	\caption{Normalized residual~$\delta_{R}$ and trajectory error~$\delta_{T}$ for the double integrator example system when $\delta_{R}$ is trustworthy. For both error types, the 10 smallest values are marked in cyan except the minimum values which are marked in red.}
	\label{fig:comparisonErrorsNoiseFree}
\end{figure}

In order to relax the conditions in Lemma~\ref{lemma:sufficient_conditions_trustworthiness} as much as possible to discuss the applicability of the MP-based IDG method, Proposition~\ref{prop:necessary_cond_trustworthiness} introduces necessary conditions for the trustworthiness of $\delta_{R}(\bm{\alpha})$.
\begin{proposition} \label{prop:necessary_cond_trustworthiness}
	If the residual error $\delta_{R}(\bm{\alpha})$ of the MP-based IDG method is trustworthy in the sense of Definition~\ref{def:trustworthiness_residual_error}, the following conditions hold:
	\begin{enumerate}
		\item The OLNE conditions of Lemma~\ref{lemma:OLNE} are necessary and sufficient\footnote{If there are more than one optimal solutions for a parameter vector $\bm{\theta}$ all these solutions need to be compared to the GT trajectories and the solution with the smallest trajectory error $\delta_T$ is used to define $\delta_T$.}.
		\item The basis functions~$\bm{\phi}_i,\forall i \in \mathcal{P}$, which yield $\delta_{T}=0$, are known.
		\item The trivial solution of the minimization of $\delta_{R}(\bm{\alpha})$ is omitted.
	\end{enumerate}
\end{proposition}
\begin{proof}
	Regarding the first condition, suppose the GT trajectories $\tilde{\bm{x}}(\cdot)$ and $\tilde{\bm{u}}(\cdot)$ are optimal w.r.t. $\bm{\phi}_i, \forall i \in \mathcal{P}$ and $\bm{\theta}^*$. If the OLNE conditions are not necessary, a global minimum of $\delta_{R}$ at $\bm{\theta}^*$ cannot be guaranteed, although $\delta_{T}(\bm{\theta}^*)=0$. Hence, the residual error would not be trustworthy. If the OLNE conditions are only necessary but not sufficient, $\delta_{R}=0$ holds when the GT trajectories are only candidates for an optimal solution w.r.t. the global minimizer $\hat{\bm{\theta}}^R$. Hence, the trajectories that are indeed optimal w.r.t. $\hat{\bm{\theta}}^R$ can be different: $\delta_{R}=0$ but $\delta_{T}(\hat{\bm{\theta}}^R)>0$. Again, the residual error would not be trustworthy.

	\begin{figure}[t]
		\centering
		\includegraphics[width=2.5in]{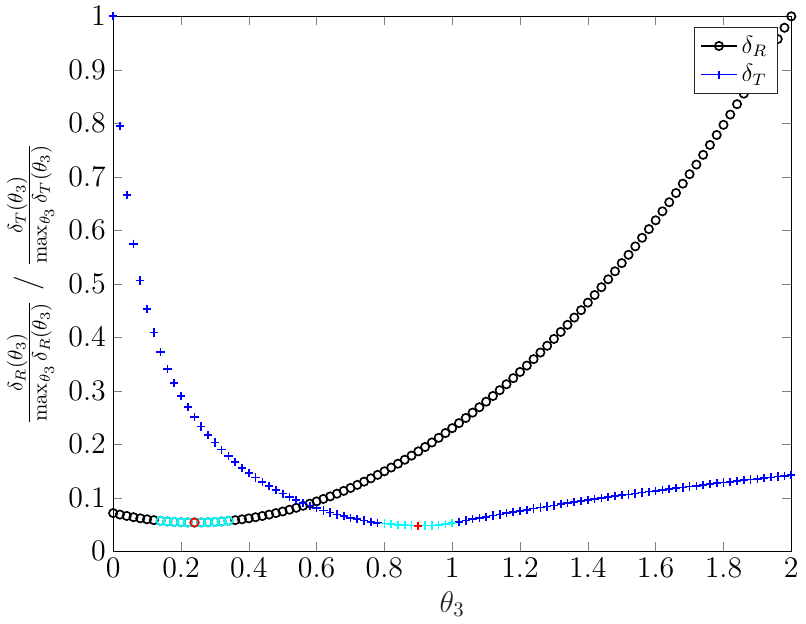}
		\caption{Normalized residual~$\delta_{R}$ and trajectory error~$\delta_{T}$ for the double integrator system when $\delta_{R}$ is not trustworthy.}
		\label{fig:comparisonErrorsBasisFct}
	\end{figure}
	\begin{figure}[t]
		\centering
		\begin{subfigure}[t]{1.65in}
			\includegraphics[width=1.65in]{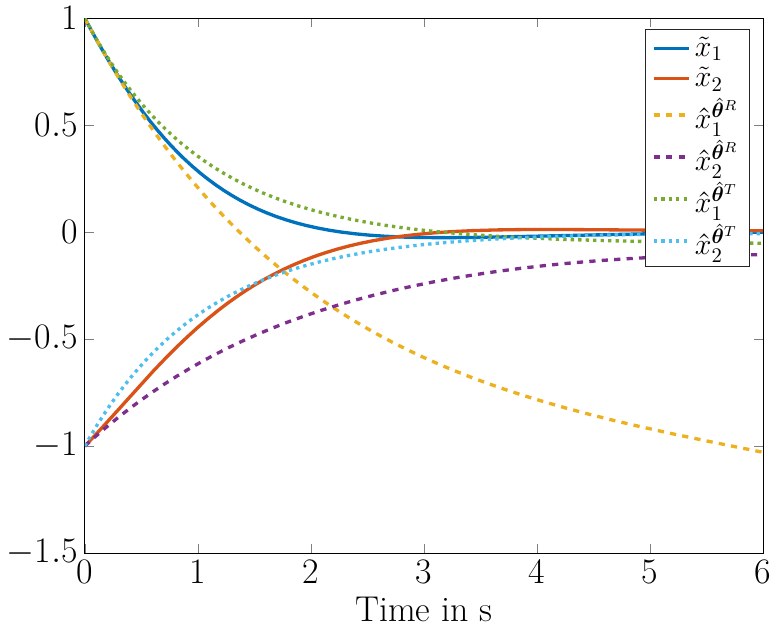}
		\end{subfigure}
		\begin{subfigure}[t]{1.65in}
			\includegraphics[width=1.65in]{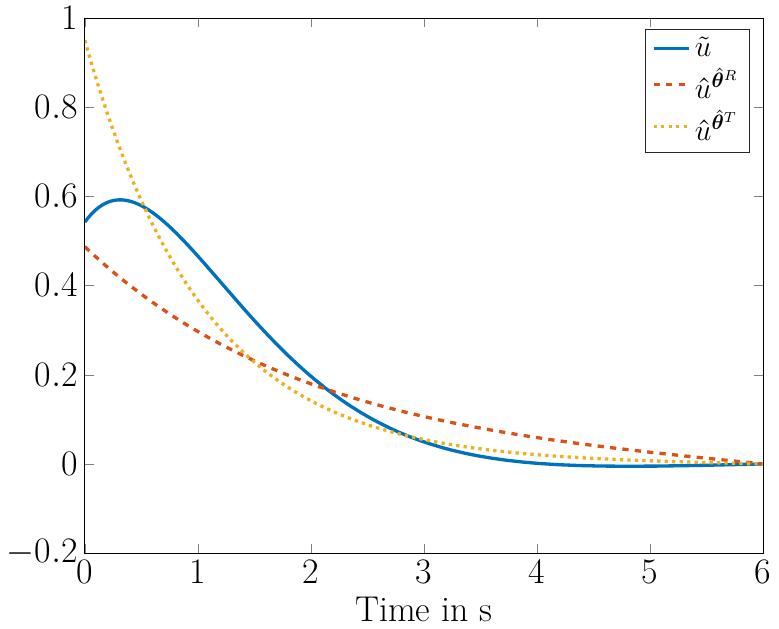}
		\end{subfigure}
		\caption{State and control trajectories of the GT (solid lines), the trajectories resulting from the global minimizer $\hat{\bm{\theta}}^{R}$ of $\delta_R$ (dashed lines) and the trajectories resulting from the global minimizer $\hat{\bm{\theta}}^{T}$ of $\delta_T$ (dotted lines) for the double integrator system when $\delta_R$ is not trustworthy.}
		\label{fig:comparisonTrajectoriesBasisFct}
	\end{figure}
	The second condition is proofed by a counterexample based on the single-player double integrator. Suppose the basis functions~$\bm{\phi}(\bm{x},u) = \mat{u^2 & x_1^2 & x_2^2 }^\top$, which yield $\delta_{T}=0$ with $\bm{\theta}^*$ are unknown and $\delta_R(\bm{\alpha})$ is set up with $\bm{\phi}(\bm{x},u) = \mat{u^2 & x_2^2 }^\top$ ($\bm{\theta}= \mat{\theta_1 & \theta_3}^\top$) instead. Fig.~\ref{fig:comparisonErrorsBasisFct} shows the residual~$\delta_R$ and the trajectory error~$\delta_T$ for varying $\theta_3$-values ($\theta_1=1$ for both errors and $\bm{\psi}(0)=\bm{\psi}^*(0)$ for $\delta_R$). The global minimizers $\hat{\bm{\theta}}^R$ and $\hat{\bm{\theta}}^T$ of both error types differ clearly and $\delta_T(\hat{\bm{\theta}}^T)>0$ at the global trajectory error minimizer. Thus, the residual error is not trustworthy in general when $\delta_T(\hat{\bm{\theta}}^T)>0$ holds with the used basis functions.
	Furthermore, Fig.~\ref{fig:comparisonTrajectoriesBasisFct} shows the trajectories resulting from the global minimizer $\hat{\bm{\theta}}^{R}$ and the ones resulting from $\hat{\bm{\theta}}^{T}$. Especially, by comparing the state trajectories $\hat{\bm{x}}^{\hat{\bm{\theta}}^{R}}(t)$ and $\hat{\bm{x}}^{\hat{\bm{\theta}}^{T}}(t)$, which do not even show similar trends, it can be concluded that in case of not fulfilled condition 2) and thus, missing trustworthiness of the residual error the MP-based IDG method cannot provide a robust parameter identification.

	With the third condition, the global minimizer $\bm{\alpha}=\bm{0}$ of $\delta_{R}$ is omitted since it yields $\delta_T>0$.
\end{proof}

\begin{remark}
	Condition 1) and 3) in Proposition~\ref{prop:necessary_cond_trustworthiness} need to be guaranteed by performing the minimization of $\delta_{R}(\bm{\alpha})$ w.r.t. the parameters $\bm{\theta}_i$ on appropriate subsets $\Theta_i \subset \mathbb{R}^{M_i}, \forall i \in \mathcal{P}$.
\end{remark}


\tdd{only Remark 1 for missing model validation (due to no falsification and perhaps through iterating the model, e.g. enhancing the basis functions, no match with at least one global minimum of the trajectory error guaranteed) - model validation with direct approach, here best possible basis function setup; Remark 2: then, we can get an approximate trustworthiness, important in practice since exact trustworthiness not reasonable in practice due to remaining errors between GT and model trajectories - note: trustworthiness also important for computation of reasonable parameters for a validated model with an indirect method (residual-error minimizers need to be from a subset of the trajectory-error minimizers)}

With the necessary conditions provided by Proposition~\ref{prop:necessary_cond_trustworthiness}, the applicability of the MP-based IDG method for model validation and robust parameter identification can be discussed. If the residual error is not trustworthy, it is not possible to validate a postulated DG model with measurement data. Firstly, falsification of the model hypothesis is not possible since $\hat{\bm{\theta}}^R$ computed by the indirect IDG method can yield $\delta_{T}(\hat{\bm{\theta}}^R)>0$ although $\delta_{T}(\hat{\bm{\theta}}^T)= 0$ holds at the global trajectory error minimizer. Secondly, it is not guaranteed that by continuously extending the number of basis functions or by using an overparameterized cost function structure $\delta_{T}(\hat{\bm{\theta}}^R)= 0$ can be once achieved although $\delta_{T}(\hat{\bm{\theta}}^T)= 0$ holds. Now, since trustworthiness necessarily depends on knowing the cost function structures (see condition 2) in Proposition~\ref{prop:necessary_cond_trustworthiness}) which in turn are unknown in case of model validation, we conclude Remark~\ref{remark:modelValidation_residualIOCIDG}.
\begin{remark} \label{remark:modelValidation_residualIOCIDG}
	The indirect IDG method based on the MP cannot be used to validate a DG model with a postulated cost function structure (set of basis functions).
\end{remark}

Remarkably, our proposed conditions are only necessary in their current form. Although we may find, e.g. via a direct IDG approach\comment{or an overparameterization strategy and the MP-based IDG method}, cost function structures with $\delta_{T}(\hat{\bm{\theta}}^T)= 0$ and $\hat{\bm{\theta}}^T =\hat{\bm{\theta}}^R$, it is not guaranteed that every global residual error minimizer $\hat{\bm{\theta}}^R$ yields $\delta_{T}(\hat{\bm{\theta}}^R)= 0$.

In practice, although the cost function structures are known or condition 2) of Proposition~\ref{prop:necessary_cond_trustworthiness} is validated with a direct IDG method, $\delta_{T}(\hat{\bm{\theta}}^T)= 0$ is typically not possible, e.g. due to measurement noise, and small errors $\delta_{T}(\hat{\bm{\theta}}^T)< \epsilon_T$ remain but are tolerable for the concrete application. Here, $\hat{\bm{\theta}}^T=\hat{\bm{\theta}}^R$ is not possible but the MP-based IDG method can provide, under certain conditions, a robust parameter identification such that $\norm{\hat{\bm{\theta}}^R - \hat{\bm{\theta}}^T}<\epsilon$ where $\epsilon$ is sufficiently small to meet the application requirement $\delta_{T}(\hat{\bm{\theta}}^R)< \epsilon_T$.
Hereto, we look at the double integrator system \eqref{eq:sys_doubleInt}, \eqref{eq:cost_doubleInt} and define as GT data the OC solution with $\bm{\theta}^*$ but add Gaussian noise: $\tilde{\bm{x}}(t) = \bm{x}^*(t) + \bm{\varepsilon}_1(t)$, $\tilde{u}(t) = u^*(t) + \varepsilon_2(t)$ ($\bm{\varepsilon}_1(t) \sim \mathcal{N}(\bm{0},0.1^2 \bm{I})$, $\varepsilon_2(t) \sim \mathcal{N}(0,0.1^2)$). We can assume that the given $\bm{\phi}$ are the best possible basis functions regarding $\delta_T$, i.e. there is no other OC model with a smaller trajectory error. Fig.~\ref{fig:comparisonErrorsNoise} shows the normalized $\delta_R$ (with $\theta_1=\theta_1^*$, $\theta_3=\theta_3^*$ and $\bm{\psi}(0)=\bm{\psi}^*(0)$) and $\delta_T$ (with $\theta_1=\theta_1^*$ and $\theta_3=\theta_3^*$) regarding different $\theta_2$-values. Although an exact match with the GT trajectories is not possible ($\delta_{T}(\hat{\bm{\theta}}^T)> 0$, $\delta_{T}(\hat{\bm{\theta}}^R)> 0$), an approximate trustworthiness, i.e. $\norm{\hat{\bm{\theta}}^R - \hat{\bm{\theta}}^T}<\epsilon$, follows, where $\epsilon$ is small enough such that $\hat{\bm{\theta}}^T$ and $\hat{\bm{\theta}}^R$ both yield nearly identical trajectories with $\delta_{T}(\hat{\bm{\theta}}^T)<\epsilon_T$ and $\delta_{T}(\hat{\bm{\theta}}^R)<\epsilon_T$. In this case, $\epsilon_T$ stems from the simulated measurement noise and is negligible.
Finally, we conclude Remark~\ref{remark:approximate_match_minima}.

\begin{remark} \label{remark:approximate_match_minima}
	If a DG model is validated, e.g. with a direct IDG approach, such that the corresponding cost function structures guarantee the smallest possible trajectory error~$\delta_T$, an approximate trustworthiness of the residual error~$\delta_R$, i.e. $\norm{\hat{\bm{\theta}}^R - \hat{\bm{\theta}}^T} < \epsilon$ can be achieved such that $\epsilon$ is sufficiently small for $\delta_{T}(\hat{\bm{\theta}}^R)<\epsilon_T$ and $\delta_{T}(\hat{\bm{\theta}}^T)<\epsilon_T$ ($\epsilon_T$ depends on the concrete application of the DG model). 
	\comment{approximate match reasonable, since the different error measures can show different global minima when the global minimum values are not necessarily zero at the same point - direct approach more robust regarding the trajectory error since its objective function is based on this error measure (can be used for re-optimizing the parameters achieved by the MP approach which in turn gives a good initialization of the bi-level approach if the basis functions are already good enough)}
\end{remark}

\begin{figure}[t]
	\centering
	\includegraphics[width=2.5in]{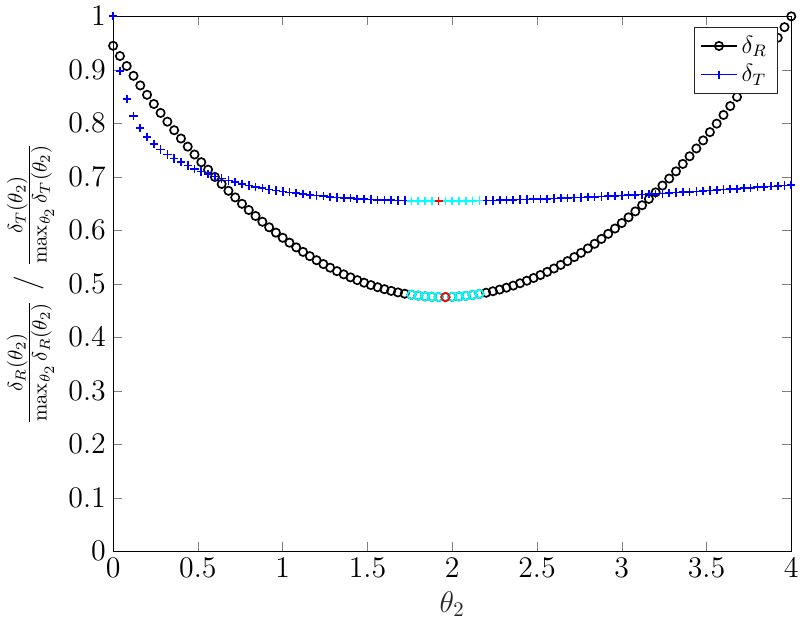}
	\caption{Normalized residual~$\delta_{R}$ and trajectory error~$\delta_{T}$ for the double integrator system when $\delta_{R}$ is approximately trustworthy.}
	\label{fig:comparisonErrorsNoise}
\end{figure}
\begin{figure}[t]
	\centering
	\begin{subfigure}[t]{1.65in}
		\includegraphics[width=1.65in]{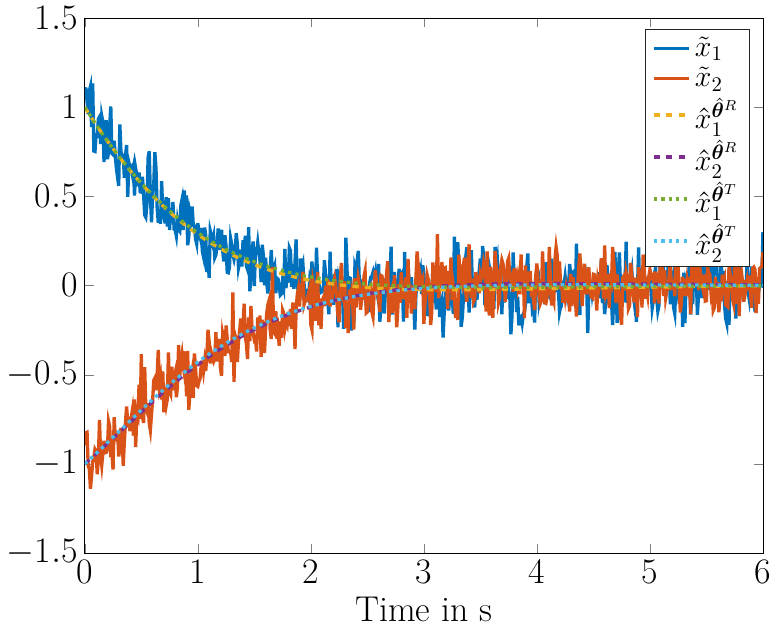}
	\end{subfigure}
	\begin{subfigure}[t]{1.65in}
		\includegraphics[width=1.65in]{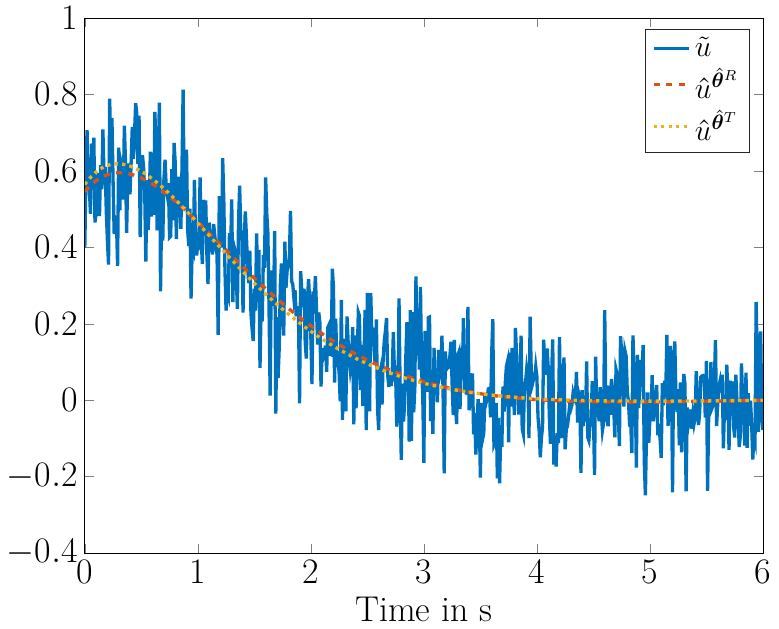}
	\end{subfigure}
	\caption{State and control trajectories of the GT, the trajectories resulting from the global minimizer $\hat{\bm{\theta}}^{R}$ and the trajectories resulting from the global minimizer $\hat{\bm{\theta}}^{T}$ for the double integrator system when $\delta_R$ is approximately trustworthy.}
	\label{fig:comparisonTrajectoriesNoise}
\end{figure}

Remark~\ref{remark:approximate_match_minima} can be beneficial in practice. Whereas validating a model assumption, e.g. for human movements, can typically be performed offline, applications of validated models often require fast or even real-time parameter identification, for example to adapt prediction models for human movements to different persons or changes in their characteristics. The high computation times of direct IDG approaches \cite{Molloy.2017} can be tolerated in offline model validation but hinder such fast/real-time parameter identifications, which however are possible with residual-based methods \cite{Inga.2021}.

\tdd{move in conclusion and clarify that currently from literature only sufficient conditions for trustworthiness (see Lemma 3) can be derived, necessary or necessary/sufficient conditions are missing}


\section{Practical Example: Identification of Multi-Agent Collision Avoidance Behavior} \label{sec:multiAgentCollAvoid}

In this section, we illustrate the consequences of Remark~\ref{remark:modelValidation_residualIOCIDG} and Remark~\ref{remark:approximate_match_minima} by analyzing the validity of a differential game model to describe the collision avoidance behavior between two humans driving mobile robots in a simulation environment.

\subsection{Differential Game Model for Multi-Agent Collision Avoidance between Mobile Robots with Human Operators} \label{subsec:diffGameModel}

The dynamics of each robot $i \in \mathcal{P}=\{1,2\}$ are given by $\dot{\bm{x}}_i = \bm{u}_i$, where $x_{i,1}$ denotes the x-position and $x_{i,2}$ the y-position of robot $i$. This yields the complete system dynamics $\dot{\bm{x}} = \mat{ \dot{\bm{x}}^\top_1 & \dot{\bm{x}}^\top_2 }^\top = \bm{u}$. The individual cost function of each human operator $i \in \mathcal{P}$ is modeled by \eqref{eq:cost_function_basisFunctions} where the basis functions~$\bm{\phi}_i(\bm{x},\bm{u})$ follow from
\begin{align} \label{eq:basisFcts_collAvoidExample}
	\bm{\lambda}_i(\bm{x}(T)) &=  \big[ (x_{i,1}(T)-x_{i,T,1})^2 \,\,\, (x_{i,2}(T)-x_{i,T,2})^2 \big]^\top, \\
	\bm{\mu}_i(\bm{x},\bm{u}_i) &= \big[ u_{i,1}^2 \,\,\, u_{i,2}^2 \,\,\, (x_{i,1}-x_{i,T,1})^2 \,\,\, (x_{i,2}-x_{i,T,2})^2 \nonumber \\ &\hphantom{=}\norm{\bm{x}_{i}-\bm{x}_j}^{-1}_2 \,\,\, \norm{\bm{x}_{i}-\bm{x}_j}^{-2}_2 \big]^\top
\end{align}
with $j \in \mathcal{P}$ but $j\neq i$. Moreover, the starting positions are given by $\bm{x}_0 = \mat{ -1\,\text{m} & -0.5\,\text{m} & 1\,\text{m} & 0\,\text{m} }^\top$ and the displayed target points by $\bm{x}_T = \mat{ 1\,\text{m} & 1\,\text{m} & -1\,\text{m} & 0\,\text{m} }^\top$.

Firstly, we apply the IDG method based on the MP to GT trajectories computed by the DG model with $T=5\,\text{s}$, $\bm{\theta}^*_1 = \mat{ 1 & 4 & 0 & 0 & 0.2 & 0 & 100 & 100 }^\top$ and $\bm{\theta}^*_2 = \mat{ 1 & 1 & 0 & 0 & 0.2 & 0 & 100 & 100 }^\top$: $\tilde{\bm{x}}(\cdot) = \bm{x}^*(\cdot)$ and $\tilde{\bm{u}}(\cdot) = \bm{u}^*(\cdot)$. Due to rewriting the terminal costs in the integrand of $J_i$ (cf. \eqref{eq:cost_function_basisFunctions}), the MP conditions are not sufficient anymore since the system dynamics are integrated into the basis functions which violates the convexity assumption on $H_i$ in Lemma~\ref{lemma:OLNE}. Hence, the first condition in Proposition~\ref{prop:necessary_cond_trustworthiness} is not fulfilled and by computing $\hat{\bm{\theta}}^R$ with the QP~\eqref{eq:residual_error_QP} with constraints $\theta_{1,1}=\theta_{2,1}=1$, $\delta_{R,\min} = \delta_R(\hat{\bm{\alpha}}^R) \approx 1.0 \cdot 10^{-7}$ and $\delta_{T,\delta_{R,\min}} = \delta_{T}(\hat{\bm{\theta}}^R) \approx 2.8 \cdot 10^{3}$ follows, although $\delta_{T,\min} = \delta_{T}(\hat{\bm{\theta}}^T)\approx0$ is possible. Therefore, $\theta_{1,7}\geq500$, $\theta_{1,8}\geq500$, $\theta_{2,7}\geq500$ and $\theta_{2,8}\geq500$ are introduced as additional constraints. If these weights of the terminal costs are beyond a certain threshold the robots reach their desired target points but further increasing them does not influence the trajectories which instead are mainly affected by the running cost weights. With these additional constraints, $\delta_{R,\min} \approx 2.8 \cdot 10^{-3}$ and $\delta_{T,\delta_{R,\min}} \approx 12.6$ result.

The Two-Point-Boundary-Value problem (TPBVP) resulting from Lemma~\ref{lemma:OLNE} to compute an OLNE is solved with the bvp4c solver in Matlab throughout this work. Since a unique solution cannot be guaranteed in general, we compute a first solution with a constant trajectory ($\bm{x}_0$) as initial guess. Then, the TPBVP is solved with a second initial guess, which is given by mirroring the trajectories of each robot of the first solution at the line between $\bm{x}_{i,0}$ and $\bm{x}_{i,T}$\footnote{This procedure is motivated by our observation that our setting as well as the DG model formulation tend to possess symmetrical OLNE. For example, in a two-robot setting where the target point of each robot is the starting point of the other robot, there is no preference whether both robots drive in their local coordinate system right or left.}. The TPBVP solution with the smallest $\delta_T$ is used to define $\delta_T$ for a specific parameter vector $\bm{\theta}$.

\subsection{Study Design} \label{subsec:studyDesign}

We conducted a simulation study where two humans controlled each a holonomic robot platform with a Playstation 4 controller in the x-y-plane in a Gazebo simulation (cf. Fig.~\ref{fig:screenshotStudy}). The task for both humans was to drive its robot in a given maximum time horizon of $5\,\text{s}$ from the starting position $\bm{x}_{i,0}$ to the displayed target point $\bm{x}_{i,T}$ (cross in the same color as the robot, yellow for robot $1$ and purple for robot $2$, see Fig.~\ref{fig:screenshotStudy}) while avoiding collision to the other robot and the test field bounds. The target point was considered as reached if it was inside the rectangular robot ($0.33\,\text{m} \times 0.31\,\text{m}$). A mutual visual starting signal for both players was given via the graphical user interface. Communication between the participants was prohibited before and during the experiment. The collision avoidance experiment was conducted with four groups with two participants each.
\begin{figure}[t]
	\centering
	\includegraphics[width=3in]{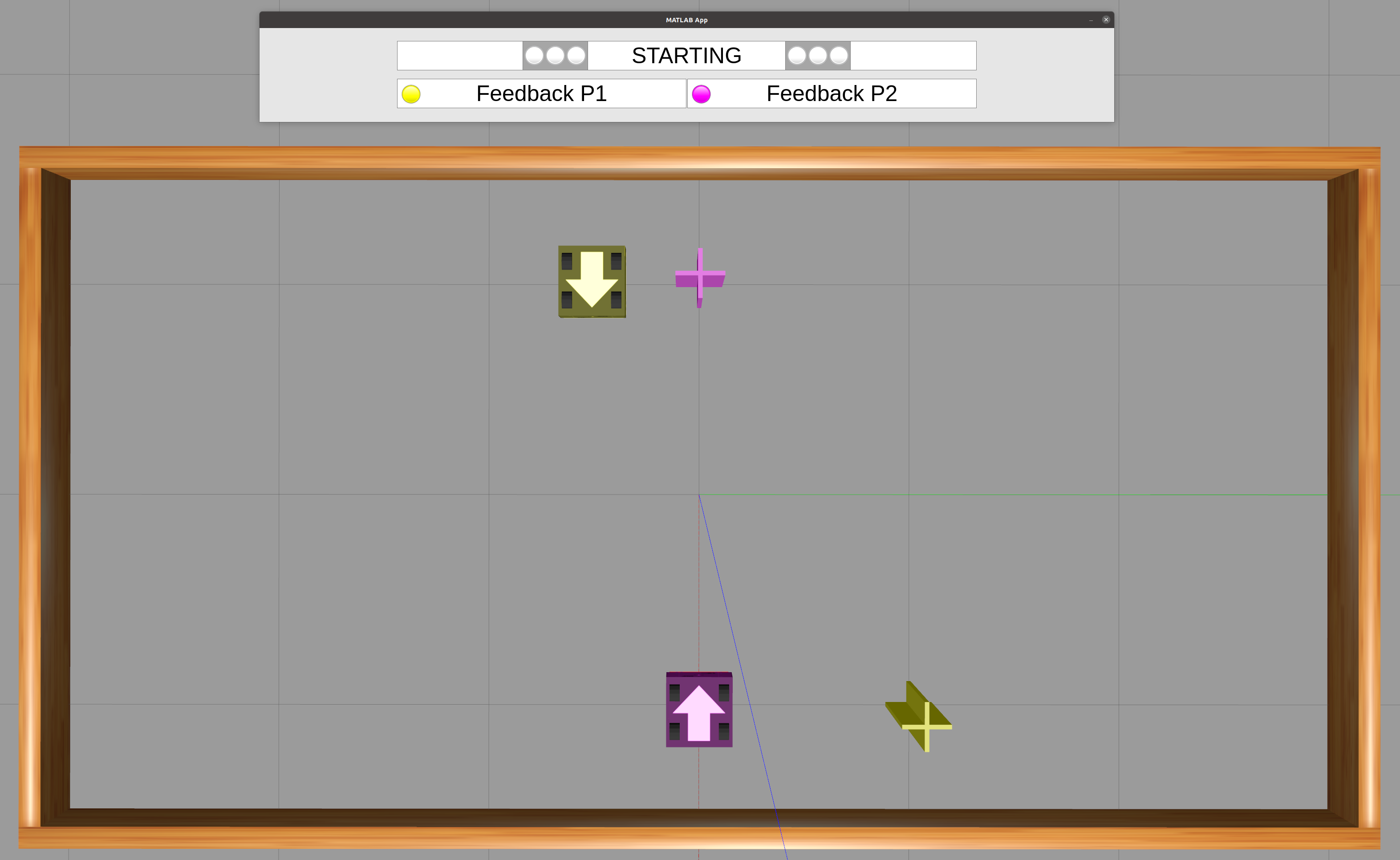}
	\caption{Screenshot of the simulation study to analyze the collision avoidance behavior between two humans.}
	\label{fig:screenshotStudy}
\end{figure}

Motivated by OC models for human movements (see e.g. \cite{Berret.2011,Mombaur.2010}), we hypothesize that the observed behavior can be described by an OLNE. The model assumption is checked by applying data-based IDG methods to the observed GT data.
Before the actual collision avoidance experiment between both participants, a familiarization phase was implemented for each participant where the participant had to drive to $10$ arbitrary, successively displayed target points. Then, each group performed $10$ trials of the collision avoidance experiment as training phase to learn the other player's behavior and then, $10$ trials to collect the GT data. The center of the robot was defined as $\bm{x}_i$ and its position was recorded with $50\,\text{Hz}$. The corresponding control signal $\bm{u}_i$ follows from numerical differentiation and a cubic spline smoothing. The start time $t_0$ of the DG model of one trial was defined as the first point in time where for one robot $\norm{\bm{u}_i}_2>0.15\,\frac{\text{m}}{\text{s}}$ holds and the final time $T$ as the first point in time where $\norm{\bm{u}_i}_2<0.1\,\frac{\text{m}}{\text{s}}, \forall i \in \mathcal{P}$. Finally, $26$ valid trials result where no collisions occurred and the target points are reached at the final time $T$. 
For the valid trials, the DG model of Section~\ref{subsec:diffGameModel} is suggested.
In order to validate this model assumption, the measured GT trajectories $\tilde{\bm{x}}(\cdot)$, $\tilde{\bm{u}}(\cdot)$ of each trial are used to compute parameters $\hat{\bm{\theta}}$. The MP-based IDG method yields $\hat{\bm{\theta}}^{R}$ at the global minimum residual error $\delta_{R,\min}$ with trajectory error $\delta_{T,\delta_{R,\min}} = \delta_{T}(\hat{\bm{\theta}}^R)$ and the direct bi-level-based method minimizing $\delta_{T}$ with the pattern search algorithm\comment{"standard" initialization: ones for all parameters, except for $\bm{\zeta}_i=750$} of the Matlab environment yields $\hat{\bm{\theta}}^{T}$ at the global minimum trajectory error $\delta_{T,\min}$. For both methods, the constraints as defined in Section~\ref{subsec:diffGameModel} are used as well as the definition of $\delta_T$. For each trial, $\bm{x}_T$ of the DG model is set to $\bm{x}_T=\tilde{\bm{x}}(T)$, where $T$ is the game time horizon determined for each trial based on the absolute velocity thresholds defined before. 
\commentd{here, we only define a criterion to reject the model assumption since we only want to examine whether the DG model is in principle valid to describe the GT data; typically, one would introduce a comparison model and require significant improvement of the DG model compared to the baseline model and a sufficiently good match with the GT data, i.e. formulation of a hypothesis and statistical tests; additionally, validation and training data would be separated}

\subsection{Study Results} \label{subsec:studyResults}
\begin{table}[t]
	\caption{IDG results for the five best and the five worst trials according to $\delta_{T,\min}$.}
	\label{table:study_results}
	\renewcommand{\arraystretch}{1.15}
	\begin{center}
		\begin{tabular}{ | c | c | c | c | c | }
			\hline
			Group & Trial Number & $\delta_{R,\min}$ & $\delta_{T,\delta_{R,\min}}$ & $\delta_{T,\min}$ \\ 
			\hline
			\hline
			 $1$ & $14$ & $23.6$ & $353.2$ & $249.8$ \\  
			 \hline
			 $2$ & $15$ & $30.3$ & $1.2 \cdot 10^{3}$ & $257.6$ \\
			 \hline
			 $4$ & $16$ & $1.1 \cdot 10^{3}$ & $1.3 \cdot 10^{3}$ & $269.5$ \\
			 \hline
			 $1$ & $17$ & $42.5$ & - & $302.8$ \\
			 \hline
			 $4$ & $11$ & $252.0$ & - & $314.3$ \\
			 \hline
			 \hline
			  $2$ & $12$ & $178.0$ & - & $558.3$ \\
			   \hline
			  $4$ & $17$ & $801.0$ & - & $574.4$ \\
			   \hline
			  $3$ & $13$ & $322.4$ & $5.0 \cdot 10^{16}$ & $601.5$ \\
			   \hline
			  $3$ & $20$ & $18.5$ & - & $730.5$ \\
			   \hline
			  $3$ & $18$ & $2.1 \cdot 10^{3}$ & - & $776.5$ \\
			   \hline
		\end{tabular}
	\end{center}
\end{table}
\begin{figure}[t]
	\centering
	\begin{subfigure}[t]{1.65in}
		\includegraphics[width=1.65in]{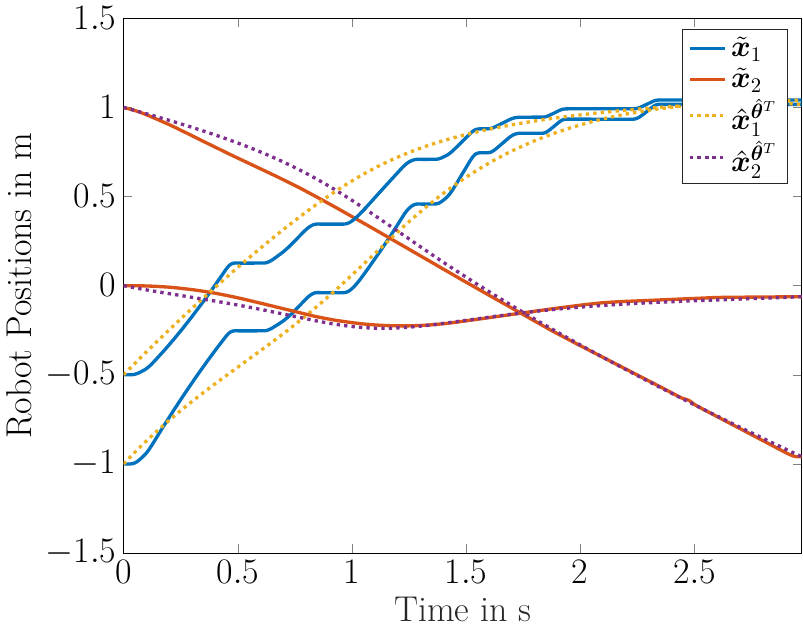}
	\end{subfigure}
	\begin{subfigure}[t]{1.65in}
		\includegraphics[width=1.65in]{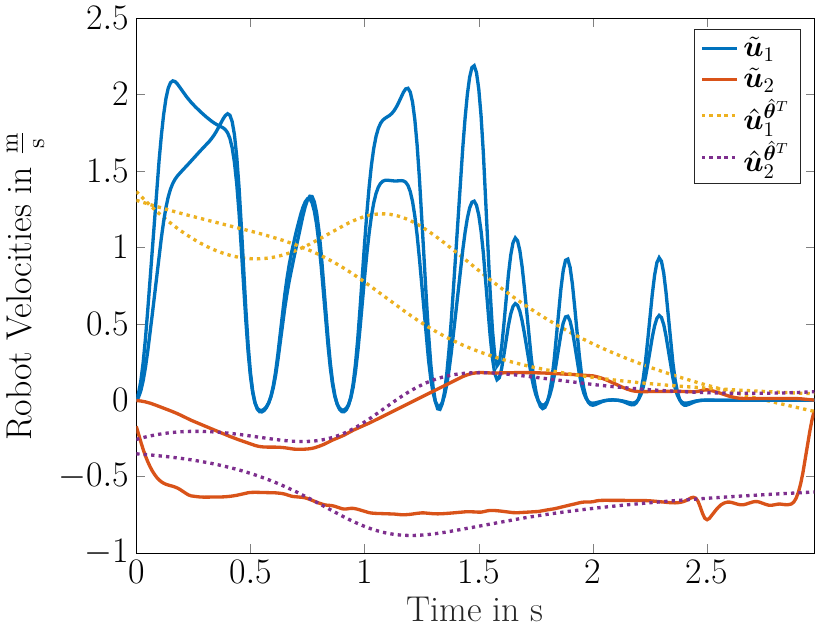}
	\end{subfigure}
	\caption{State and control trajectories of the GT and the direct bi-level-based identification for group $2$/trial $15$.}
	\label{fig:comparisonTrajectoriesCollAvoidBest}
\end{figure}
\begin{figure}[t]
	\centering
	\begin{subfigure}[t]{1.65in}
		\includegraphics[width=1.65in]{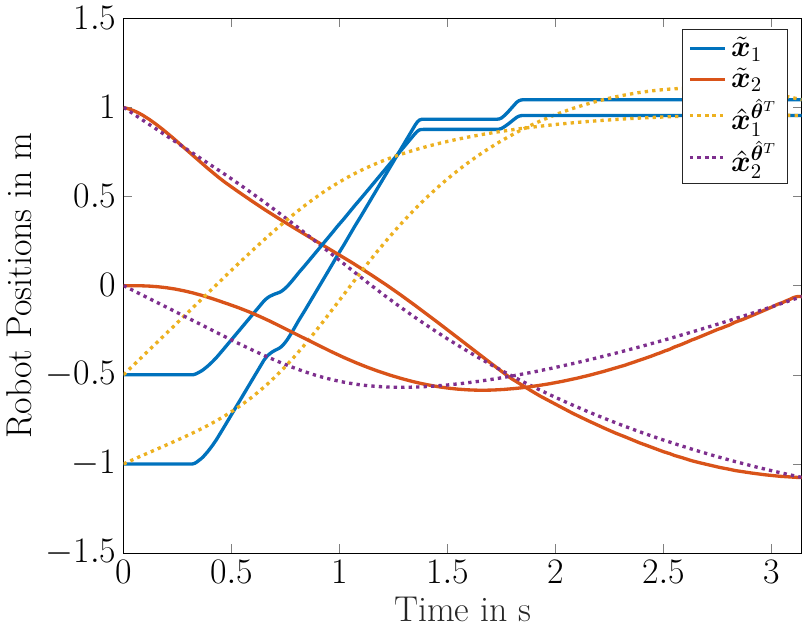}
	\end{subfigure}
	\begin{subfigure}[t]{1.65in}
		\includegraphics[width=1.65in]{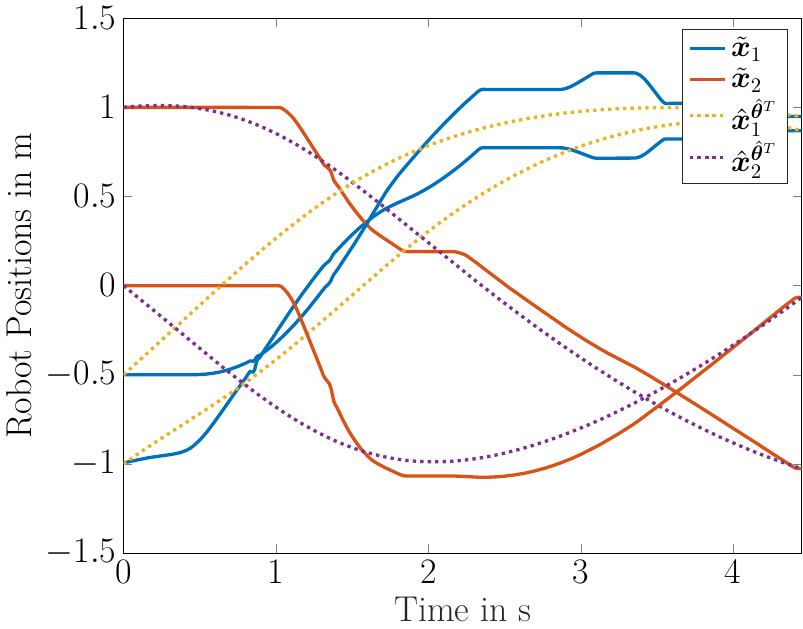}
	\end{subfigure}
	\caption{State trajectories of the GT and the bi-level-based identification for group $2$/trial $17$ ($\delta_{T,\min}\approx434.8$) and group $4$/trial $17$ ($\delta_{T,\min}\approx574.4$).}
	\label{fig:comparisonTrajectoriesCollAvoidThreshold}
\end{figure}
\begin{figure}[t]
	\centering
	\begin{subfigure}[t]{1.65in}
		\includegraphics[width=1.65in]{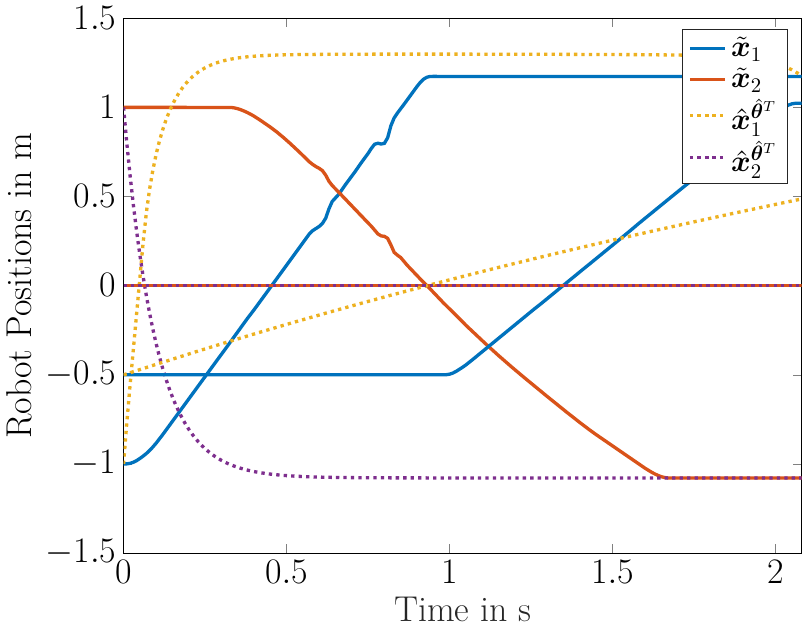}
	\end{subfigure}
	\begin{subfigure}[t]{1.65in}
		\includegraphics[width=1.65in]{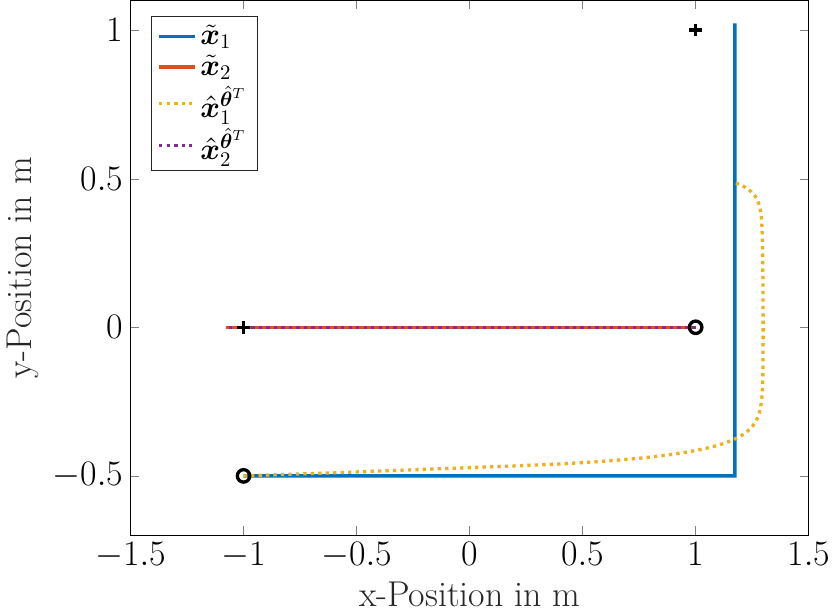}
	\end{subfigure}
	\caption{State trajectories and paths of the GT and the direct bi-level-based identification for group $3$/trial $18$.}
	\label{fig:comparisonTrajectoriesCollAvoidWorst}
\end{figure}
Table~\ref{table:study_results} shows the achieved results of the IDG evaluation described in Section~\ref{subsec:studyDesign}. The trials are sorted according to $\delta_{T,\min}$ and for clarity, only the five best and five worst trials are shown\comment{in order to keep computation time tractable, the maximum number of iterations for the pattern search algorithm was set to $200$; hence, most likely even smaller $\delta_{T,\min}$ values can be reached since in all cases the pattern search stopped prematurely due to reaching the maximum number of iterations}.
Except for group $1$/trial $14$, in all cases $\delta_{T,\delta_{R,\min}}$ and $\delta_{T,\min}$ differ by at least one order of magnitude and in $19$ trials an OLNE cannot even be computed with the parameters $\hat{\bm{\theta}}^{R}$ (see missing values for $\delta_{T,\delta_{R,\min}}$ in Table~\ref{table:study_results}). This highlights that the residual error is not trustworthy in this practical example due to the unknown optimal basis function configuration. Thus, according to Remark~\ref{remark:modelValidation_residualIOCIDG}, the MP-based IDG method is not suitable to evaluate the DG model assumption. The minimum, mean and maximum computation times for $\hat{\bm{\theta}}^{R}$ are $0.33\,\text{s}$, $0.55\,\text{s}$ and $1.19\,\text{s}$, respectively. In contrast, the minimum, mean and maximum computation times for $\hat{\bm{\theta}}^{T}$ are $2.30 \cdot 10^{3}\,\text{s}$, $4.41 \cdot 10^{3}\,\text{s}$ and $6.67 \cdot 10^{3}\,\text{s}$, respectively, and thus, nearly four orders of magnitude higher\footnote{The computation times were achieved with a Ryzen 9 5950X. The $16$ available cores of the CPU were used for parallel implementation of the pattern search algorithm.}.

By evaluating $\delta_{T,\min}$ for the valid trials, it needs to be concluded that the proposed DG model is in its current form not valid to describe the human collision avoidance behavior. Considering the result for the control trajectories in Fig.~\ref{fig:comparisonTrajectoriesCollAvoidBest}, where $\delta_{T,\min} \approx 250$, the DG model fails to describe $\tilde{\bm{u}}_1(t)$. Hence, $\delta_{T,\min} \ll 250$ would be necessary to accept the model assumption. Remarkably, when the DG model is used to generate the GT trajectories (see Section~\ref{subsec:diffGameModel}), $\delta_{T,\min}\approx 13$ is achieved and results in a match between estimated and GT trajectories.

\subsection{Discussion} \label{subsec:discussion_studyResults}
The main reason for the relative high trajectory errors, follows from the GT control trajectories. Like in Fig.~\ref{fig:comparisonTrajectoriesCollAvoidBest} for $\tilde{\bm{u}}_1(t)$, we observe in all trials concatenations of several single peaked bell-shapes. This indicates that the participant has chosen a sequence of intermediate target points building the trajectory to the desired displayed target point. The intermediate target points were chosen based on the observations of the other player's behavior. Since the DG model extends OC models for single human movements, which describe one-shot movements to a single goal with a single peaked bell-shaped velocity profile, it gets clear that the DG defined in Section~\ref{subsec:diffGameModel} with one target point fails in describing the velocity profiles of the GT data with several peaks, which occur at asynchronous points in time between both players. 
The multi-agent collision avoidance behavior between two humans seems to result from a rapid interplay between two sensorimotor control levels, the action ("movement to a single goal") and the decision level ("decision on the actual goal"). Understanding the features of sensorimotor control levels regarding human decision making and their interfaces to the action level as well as suitable mathematical models are still open questions \cite{Gallivan.2018,Schneider.2022}\comment{in a future model, able to describe this rapid interplay, for the action level part a bounded (most likely reasonable, although here saturation in bounds not really observed) FNE should be chosen, incl. also a model for the human biomechanics}\comment{question: is DG model really necessary then in this case if the conflicts in this situation with a rapid interplay between decision and action level are solved by negotiate a series of intermediate target points building collision-free trajectories? however, in case of physically-coupled situation a DG model still important to describe sharing control effort}\comment{small adjustment in study design (if experimental setup is used in future work): control only x-y-positions with HMI}. 

Although the specific sensorimotor control features of the human control signals cannot be adequately described by the DG model, it can still be sufficiently good for a concrete application where a small state trajectory error is sufficient, e.g. planning collision-free trajectories for an automated mobile robot\comment{online identification of collision avoidance behavior regarding position or offline identification of this collision avoidance behavior in similar settings (start/end points, number of partners) - then, plan collision-free trajectory for automated mobile platform}. Fig.~\ref{fig:comparisonTrajectoriesCollAvoidBest} qualitatively shows a nearly perfect match between $\hat{\bm{x}}^{\hat{\bm{\theta}}^T}(\cdot)$ and $\tilde{\bm{x}}(\cdot)$ with a state trajectory error $\delta^{\bm{x}}_{T,\min}\approx 63.8$. Similar results follow for the other trials of the groups $1$, $2$ and $4$ (see e.g. Fig.~\ref{fig:comparisonTrajectoriesCollAvoidThreshold}). In Fig.~\ref{fig:comparisonTrajectoriesCollAvoidThreshold}, the higher trajectory and state trajectory errors ($\delta^{\bm{x}}_{T,\min}\approx142.5$ for group $2$/trial $17$, $\delta^{\bm{x}}_{T,\min}\approx258.5$ for group $4$/trial $17$) result from the reaction time of player $1$ and $2$, respectively. Reaction times are not integrated in the DG model and via the definition of the starting time $t_0$ based on the absolute velocity threshold of the fastest robot only the reaction time of the fastest player can be compensated. Despite this problem, in Fig.~\ref{fig:comparisonTrajectoriesCollAvoidThreshold}, the identified trajectories still predict the trend of the GT state trajectories well. This conclusion only does not apply to some trials of group $3$ which yield the worst results in Table~\ref{table:study_results} (see e.g. Fig.~\ref{fig:comparisonTrajectoriesCollAvoidWorst}). Here, one participant (player $1$) solved the collision avoidance problem by choosing an intermediate target point at $\mat{ 1.2\,\text{m} & -0.5\,\text{m} }^\top$ and bypassed the other player completely. The movements decompose into one movement to a single goal for player $2$ and two movements to single goals for player $1$, which again cannot be described by the DG model in its current form.

With these practical results, we can also illustrate Remark~\ref{remark:approximate_match_minima}. Although $\delta_{T,\min}$ may be sufficiently small for an application to compute collision-free trajectories, i.e. $\delta_{T,\min} < \epsilon_T$, a robust parameter identification with the MP-based IDG method is not given and DG parameters need to be computed by a direct IDG approach. In contrast to the noisy GT trajectories of the double integrator example, the postulated basis functions in the DG model do not provide the best DG fit to the GT data: other basis functions could lead to smaller trajectory errors $\delta_{T}$.

\section{Conclusion} \label{sec:conclusion}
The main assumption in residual-based IDG methods to find parameters which yield trajectories matching the observed ones is that the given trajectories are the solution of a DG problem where the cost function structure is known, e.g. the set of basis functions. In applications of IDG methods, like the identification of unknown agents, e.g. humans, this assumption cannot be guaranteed. In this paper, we introduce and prove necessary conditions for the trustworthiness of the residual error derived from the MP. If the residual error is trustworthy, the MP-based IDG method yields the best possible parameters regarding the error to the observed GT trajectories with the used cost function structure. However, since the knowledge of this cost function structure is necessary for the trustworthiness, we conclude that the MP-based IDG method cannot be used to validate DG models in practice. We illustrate this problem by validating a DG model for the collision avoidance behavior between two mobile robots with human operators. Since the humans' cost function structures are unknown, the residual-based IDG approach fails in determining a robust parameter identification for a postulated set of basis functions. However, a direct IDG method is able to find parameters which yield in most cases state trajectories sufficiently close to the GT ones. We strongly suspect that our findings generalize to all residual-based IDG methods, which we intend to show in future work.


\bibliographystyle{IEEEtran}

\bibliography{IEEEabrv,References3}

\end{document}